\documentclass[11pt, reqno]{amsart}
\usepackage{amsmath, amsthm, amsfonts, amssymb, amsbsy, bigstrut, graphicx, enumerate,  upref, longtable, comment, booktabs, array, caption, subcaption}

\usepackage{pstricks}
\usepackage{times}

\usepackage[numbers, sort&compress]{natbib}

\usepackage[breaklinks]{hyperref}

\newcommand{\mc}{\mathcal{C}}

\newtheorem{thm}{Theorem}[section]
\newtheorem{lmm}[thm]{Lemma}

\newtheorem{prop}[thm]{Proposition}

\theoremstyle{definition}
\newtheorem{remark}[thm]{Remark}

\newcommand{\ee}{\mathbb{E}}

\newcommand{\ma}{\mathcal{A}}

\newcommand{\pp}{\mathbb{P}}

\newcommand{\rr}{\mathbb{R}}

\newcommand{\ve}{\varepsilon}

\newcommand{\mb}{\mathcal{B}}

\numberwithin{equation}{section}

\usepackage[utf8]{inputenc}
\usepackage{amsmath}
\usepackage{amsfonts}
\usepackage{bbm}
\usepackage{mathtools}
\usepackage{tikz}
\usetikzlibrary{decorations.pathreplacing}
\usepackage{amsthm}
\usepackage[english]{babel}
\usepackage{fancyhdr}
\usepackage{amssymb}
\usepackage{enumitem}
\usepackage{qtree}
\usepackage{mathrsfs}

\renewcommand{\tilde}{\widetilde}

\begin{document}
\title{Isomorphisms between random graphs}
\author{Sourav Chatterjee}
\author{Persi Diaconis}

\address{Departments of mathematics and statistics, Stanford University}
\email{souravc@stanford.edu}
\email{diaconis@math.stanford.edu}
\thanks{S.~Chatterjee was partially supported by NSF grants DMS-1855484 and DMS-2113242}
\thanks{P.~Diaconis was partially supported by NSF grant DMS-1954042}
\keywords{Random graph, graph isomorphism, concentration inequality}
\subjclass[2020]{05C80, 05C60, 60C05}

\begin{abstract}
Consider two independent Erd\H{o}s--R\'enyi $G(N,1/2)$ graphs. We show that with probability tending to $1$ as $N\to\infty$, the largest induced isomorphic subgraph has size either $\lfloor x_N-\ve_N\rfloor$ or $\lfloor x_N+\ve_N \rfloor$, where $x_N=4\log_2 N -2 \log_2 \log_2 N - 2\log_2(4/e)+1$ and $\ve_N = (4\log_2 N)^{-1/2}$. Using similar techniques, we also show that if $\Gamma_1$ and $\Gamma_2$ are independent $G(n,1/2)$ and $G(N,1/2)$ random graphs, then $\Gamma_2$ contains an isomorphic copy of $\Gamma_1$ as an induced subgraph with high probability if $n\le \lfloor y_N - \ve_N \rfloor$ and does not contain an isomorphic copy of $\Gamma_1$ as an induced subgraph with high probability if $n>\lfloor y_N+\ve_N \rfloor$, where $y_N=2\log_2 N+1$ and $\ve_N$ is as above.
\end{abstract}

\maketitle

\section{Introduction}
This paper has two motivations. First, in giving expository talks about the Rado graph, we asked: ``Let $\Gamma_1$ and $\Gamma_2$ be independent $G(N, 1/2)$. What's the chance they are isomorphic? Small? How small? Less than $N!2^{-{N\choose2}}$. So, when $N=100$, less than $10^{-1300}$. Now suppose $N=\infty$ (a $G(\infty,1/2)$ graph is called a Rado graph). The chance that $\Gamma_1 \cong \Gamma_2$ is one!''  To help think about this seeming discontinuity, we asked, for finite $N$, how large is the largest induced isomorphic subgraph of $\Gamma_1$ and $\Gamma_2$. 

\begin{thm}\label{thm1}
Let $\Gamma_1$ and $\Gamma_2$ be independent $G(N,1/2)$ graphs. Let $L_N$ be the size of the largest induced isomorphic subgraph. Then, with probability tending to one as $N\to\infty$, $L_N$ is either $\lfloor x_N-\ve_N\rfloor$ or $\lfloor x_N+\ve_N \rfloor$, where $x_N = 4\log_2 N -2 \log_2 \log_2 N - 2\log_2(4/e)+1$ and $\ve_N = (4\log_2 N)^{-1/2}$.
\end{thm}
To interpret the above result, note that $\ve_N\le 1/2$ when $N\ge 2$. This implies that $\lfloor x_N -\ve_N\rfloor$ and $\lfloor x_N + \ve_N \rfloor$ are either the same integer, or differ exactly by one. Thus, $L_N$ is either concentrated at one point, or at two consecutive points, depending on $N$. The latter case can happen only when $x_N$ is close to an integer (specifically, within less than $\ve_N$ of an integer).

Incidentally, our proof does not rule out the possibility that $L_N$ is asymptotically concentrated on one point, and not two. Proving or disproving this would require a more delicate analysis of certain remainder terms than what we currently have (see the remark at the end of Section \ref{thm1proof}). Concentration on two points, although rare, is not unprecedented. For a classical example, see~\cite[Theorem 5.1]{ddf}.

Simulations run by Ciaran McCreesh and James Trimble using the McSplit algorithm~\cite{mpt} (personal communication) indicate that the asymptotics of Theorem~\ref{thm1} kick in for $N$ as small as $31$ (with a sample of size $1$). For $N = 31$, we have $x_N \doteq 15.08$ and $\ve_N \doteq 0.23$, and in the simulation, $L_N$ turned out to be~$14$.  (Theorem \ref{thm1} says that with high probability, $L_N\in \{14, 15\}$.) And the results continue to match the prediction of Theorem \ref{thm1} (sometimes off by one) all the way up to $N=45$, when the task starts  approaching computational limits (see Table \ref{simtab}).

\begin{table}[t]
\begin{scriptsize}
\caption{Results from simulating one instance of $G(N,1/2)$. \label{simtab}}
\begin{tabular}{lrrrrrrrrrrrrrrr}
\toprule
$N$ & $31$ & $32$ & $33$ & $34$ & $35$ & $36$ & $37$ & $38$ & $39$ & $40$ & $41$ & $42$ & $43$ & $44$ & $45$ \\
\midrule
$\lfloor x_N - \ve_N\rfloor$ & $14$ & $15$ & $15$ & $15$ & $15$ & $15$ & $15$ & $15$ & $16$ & $16$ & $16$ & $16$ & $16$ & $16$ & $16$\\
$\lfloor x_N + \ve_N\rfloor$ & $15$ & $15$ & $15$ & $15$ & $15$ & $16$ & $16$ & $16$ & $16$ & $16$ & $16$ & $16$ & $16$ & $17$ & $17$ \\
$L_N$ & $14$ & $14$ & $14$ & $14$ & $15$ & $15$ & $15$ & $15$ & $15$ & $15$ & $15$ & $15$ & $16$ & $16$ & $16$\\
\bottomrule
\end{tabular}
\end{scriptsize}
\end{table}

Theorem~\ref{thm1} is proved in Section \ref{thm1proof}, which gives some background.  The argument requires some work beyond the usual second moment method. 

Our second motivation for the problems studied here came from asking experts (we were sure that Theorem \ref{thm1} was in the literature). Svante Janson had just been asked the second question in the abstract (the one on subgraph isomorphism, described below) by Don Knuth. Subgraph isomorphism is a basic question in constraint satisfaction. Developers of programs had used $G(n,1/2)$, $G(N,1/2)$ as the center of extensive tests comparing algorithms. They had observed a sharp transition: When $n = 15$, $N=150$, the chance of finding an induced copy of $\Gamma_1$ in $\Gamma_2$ is close to $1$. When $n =16$, $N=150$, the chance is close to zero. Our theorem predicts this. To state a sharp result, let $P(n,N)$ be the probability that a random $G(N,1/2)$ graph contains an induced copy of a random $G(n,1/2)$ graph. The following theorem shows that $P(n,N)$ drops from $1$ to $0$ in a window of size~$\le 2$.
\begin{thm}\label{thm2}
With above notation, as $N\to \infty$, $P(\lfloor y_N - \ve_N \rfloor,N)\to 1$ and $P(\lfloor y_N + \ve_N \rfloor + 1,N)\to 0$, where $y_N = 2\log_2 N + 1$ and $\ve_N=(4\log_2N)^{-1/2}$. 
\end{thm}

For example, when $N= 150$, $y_N \doteq 15.46$ and $\ve_N \doteq 0.19$. So the theorem says we expect $P(15, 150) \doteq1$ and $P(16,150)\doteq0$, matching explicit computation~\cite{mc}. 

Incidentally, Theorem \ref{thm2} is related to the classic problem of figuring out the size of the smallest graph that contains all graphs on $n$ vertices as induced subgraphs. After many years of partial progress, this problem was recently settled by Noga Alon~\cite{alon}, who showed that the minimum $N$ is $(1+o(1))2^{(n-1)/2}$. The key step in Alon's proof is to show that if ${N\choose n} 2^{-{n\choose 2}} = \lambda$, then the probability that $G(N,1/2)$ contains all graphs of size $n$ as induced subgraphs is $(1-e^{-\lambda})^2 +o(1)$ uniformly in $n$ as $N\to\infty$. Without checking all the details, it seems to us that Alon's result  shows that the quantity $\tilde{P}(n,N) := \pp(G(N,1/2) \textup{ contains all graphs of size $n$})$ drops from $1$ to $0$ in a window of size $2$. This window is slightly to the left of our window, at $2\log_2 N - 2 \log_2\log_2 N + 2\log_2(e/2) + 1$. The reason is that the main contributor to $\tilde{P}(n,N)$ is the probability that $G(N,1/2)$ contains a copy of $K_n$ (the complete graph of size $n$), which is the `hardest' subgraph to contain. In particular, it is harder to contain $K_n$ than it is to contain a copy of $G(n,1/2)$, which is why Alon's window is shifted by $O(\log \log N)$ to the left. 

The fact that the size of the largest clique is concentrated at the above point is known from old work of Matula~\cite{matula1, matula2} and Bollob\'as and Erd\H{o}s~\cite{be76}.  Several other graph invariants have concentrated distributions but others can be proved to be `non-concentrated' at a finite number of points. For a survey and fascinating work on the chromatic number, see \cite{heckel}.

Theorem \ref{thm2} is proved in Section \ref{thm2proof},  which begins with a literature review on subgraph isomorphism. The final section has remarks and open problems.

\subsection*{Acknowledgements.} We thank Maryanthe Malliaris and Peter Cameron for teaching us about the Rado graph, and Jacob Fox, Benny Sudakov, and Noga Alon for useful feedback and references. We thank Svante Janson and Don Knuth for pointing us to subgraph isomorphism and for careful technical reading of our paper, and several useful suggestions. We thank Ciaran McCreesh and James Trimble for running a set of very helpful simulations using their McSplit algorithm, and Don Knuth for facilitating our communication with them. Lastly, we thank the two anonymous referees for several helpful comments.

\section{Isomorphic graphs}\label{thm1proof}
Because two $G(\infty,1/2)$ graphs are isomorphic with probability $1$, a random $R$ from $G(\infty,1/2)$ has come to be called {\it the} random graph. Non-random models for $R$ abound: Let the vertex set be $\{0,1,2,\ldots\}$ and put an undirected edge from $i$ to $j$ if $i<j$ and the $i^{\textup{th}}$ binary digit of $j$ is a $1$ (labeling the rightmost digit as the zeroth digit). So $i=0$ is connected to all odd $j$, $i=1$ is connected to $j \equiv 2$ or $3$ (mod $4$) and also to $0$, and so on. The amazing properties of $R$ are beautifully exposited (and proved) in Peter Cameron's lovely article \cite{cameron}. 

Logicians have developed facts about $R$ (and higher cardinality versions). After all, a graph is just a symmetric relation. Some of this is used in Diaconis and Malliaris~\cite{diaconismalliaris} to show that various algebraic problems are intractable because an associated commuting graph contains $R$ as an induced subgraph. 

The discontinuity between finite and infinite $N$ is jarring. Aren't infinite limits supposed to shed light on large finite $N$? Exploring this led to Theorem \ref{thm1}. 

A related question is pick an Erd\H{o}s--R\'enyi graph and ask for the largest $k$ such that it contains two disjoint isomorphic induced subgraphs with $k$ vertices. A combinatorial application of this problem is studied in \cite{lee}. Similar questions can be asked for other combinatorial objects (largest disjoint isomorphic subtrees in a tree, largest disjoint order isomorphic pair of sub-permutations in a permutation --- see~\cite{dudek} for applications of such problems). 

In the remainder of this section, we prove Theorem \ref{thm1}. Although we use log base $2$ in the statement of the theorem, we will work with log base $e$ throughout the proof, which will be denoted by $\log$, as usual. The proofs of both Theorems~\ref{thm1} and~\ref{thm2} make use of the following technical result. Take any $1\le m\le n$. Let $\Gamma$ be a $G(n,1/2)$ random graph. Let $X(i,j)$ be the indicator that $\{i,j\}$ is an edge in $\Gamma$. Let
\begin{align}\label{phidef}
\phi(m,n) := \sum_{\pi\in S_n} \pp(X(i,j) = X(\pi(i), \pi(j)) \textup{ for all } 1\le i<j\le m),
\end{align}
where the right side is interpreted as $n!$ if $m=0$ or $m=1$. 
\begin{prop}\label{phiprop}
There are positive universal constants $K_1$ and $K_2$ such that for all $n\ge 1$ and $2n/3\le m\le n$, 
\[
\phi(m,n) \le K_1 e^{K_2(n-m)\log (n-m)},
\]
where $0\log 0$ is interpreted as $0$.
\end{prop}
The proof of this proposition requires two lemmas. 
\begin{lmm}\label{fixedlmm}
Let $\pi\in S_n$. Let $A\subseteq \{1,\ldots,n\}$ be a set such that $\pi(i)\ne i$ for each $i\in A$. Then there exists $B\subseteq A$, such that $|B|\ge |A|/3$, and $\pi(i)\notin B$ for each $i\in B$.
\end{lmm}
\begin{proof}
Let $C = (c_1,\ldots,c_k)$ be a cycle in $\pi$. In the following, we will sometimes treat $C$ as the set $\{c_1,\ldots,c_k\}$. First, suppose that $k$ is even. Let $A_1 := A\cap \{c_1,c_3,\ldots,c_{k-1}\}$ and $A_2 := A\cap \{c_2,c_4,\ldots,c_k\}$. Let $A'$ be the larger of these two sets. Then $|A'|\ge \frac{1}{2}|A\cap C|$ and $\pi(i)\in C\setminus A'$ for each $i\in A'$. Next, suppose that $k$ is odd. If $k=1$, then $A$ does not intersect $C$ because $A$ contains no fixed point of $\pi$. In this case, let $A':=\emptyset$. If $k\ge 3$ (and odd), let $A_1 := A\cap \{c_1,c_3,\ldots, c_{k-2}\}$ and $A_2 := A\cap \{c_2,c_4,\ldots,c_{k-1}\}$. Let $A'$ be the larger of these two sets. Then again we have that $\pi(i)\in C\setminus A'$ for each $i\in A'$, and 
\begin{align}\label{aprime}
|A'| \ge \frac{1}{2}|A\cap \{c_1,\ldots,c_{k-1}\}|.
\end{align}
Now, if $C\subseteq A$, then by the above inequality and the fact that $k\ge 3$, we get
\[
|A'| \ge \frac{k-1}{2} \ge \frac{k}{3} = \frac{1}{3}|A\cap C|.
\]
On the other hand, if there is some element of $C$ that is not in $A$, then we may assume without loss of generality that $c_k\notin A$, because the cycle $C$ can be alternatively represented as $(c_l,c_{l+1},\ldots, c_k, c_1,c_2\ldots,c_{l-1})$ for any $l$. Thus, in this case,~\eqref{aprime} gives
\[
|A'|\ge \frac{1}{2}|A\cap C|. 
\]
To summarize, given any cycle $C$, we have constructed a set $A'\subseteq A\cap C$, such that $|A'|\ge \frac{1}{3}|A\cap C|$, and $\pi(i)\in C\setminus A'$ for each $i\in A'$. Let $B$ be the union of $A'$ over all cycles $C$. It is easy to see that this $B$ satisfies the two required properties. 
\end{proof}
\begin{lmm}\label{problmm}
Take any $\pi\in S_n$. Let $k := |\{i\le m: \pi(i)=i\}|$. Then
\begin{align*}
\pp(X(i,j) = X(\pi(i), \pi(j)) \textup{ for all } 1\le i < j \le m) &\le C_1e^{-C_2(m-k)m},
\end{align*}
where $C_1$ and $C_2$ are positive universal constants.
\end{lmm}
\begin{proof}
Let $F:= \{i\le m: \pi(i)=i\}$ and $A := \{i\le m: \pi(i)\ne i\}$, so that $|F|=k$ and $|A|=m-k$. By Lemma \ref{fixedlmm}, there exists $B\subseteq A$ such that $|B|\ge \frac{1}{3}|A| = \frac{1}{3}(m-k)$, and $\pi(i)\notin B$ for all $i\in B$. Moreover, since $\pi(\pi(i)) \ne \pi(i)$ for $i\in B$ (because otherwise, $\pi(i)=i\in B$), we conclude that $\pi(i)\notin F$. By independence of edges, this gives
\begin{align*}
&\pp(X(i,j) = X(\pi(i), \pi(j)) \textup{ for all } 1\le i< j\le m)\\
&\le \pp(X(i,j) = X(\pi(i), \pi(j)) \textup{ for all } i\in B, \, j\in F)\\
&= \pp(X(i,j) = X(\pi(i), j) \textup{ for all } i\in B, \, j\in F)\\
&= 2^{-|B||F|} \le 2^{-\frac{1}{3}(m-k)k}.
\end{align*}
On the other hand, since $\pi(B) \cap B = \emptyset$, we  have
\begin{align*}
&\pp(X(i,j) = X(\pi(i), \pi(j)) \textup{ for all } 1\le i< j\le m)\\
&\le \pp(X(i,j) = X(\pi(i), \pi(j)) \textup{ for all } i, j\in B, \, i< j)\\
&= 2^{-{|B|\choose 2}}\le 2^{-\frac{1}{36}(m-k)^2 + \frac{1}{4}}, 
\end{align*}
where the last inequality holds because $|B|\ge \frac{1}{3}(m-k)$ and ${a\choose 2} \ge \frac{1}{4}(a^2-1)$ for any nonnegative integer $a$. The proof is now completed by combining the two bounds (e.g., by taking a suitable weighted geometric mean). 
\end{proof}

\begin{proof}[Proof of Proposition \ref{phiprop}]
Throughout this proof, $C_1,C_2,\ldots$ will denote positive universal constants. Two of these, $C_1$ and $C_2$, are already fixed from Lemma \ref{problmm}. For $k=0,\ldots, m$, let $T_k$ be the set of all $\pi\in S_n$ such that $|\{i\le m: \pi(i)=i\}|=k$. Then by Lemma~\ref{problmm}, 
\begin{align}
\phi(m,n) &= \sum_{k=0}^m \sum_{\pi\in T_k} \pp(X(i,j) = X(\pi(i), \pi(j)) \textup{ for all } 1\le i< j\le m) \notag\\
&\le C_1\sum_{k=0}^m |T_k| e^{-C_2(m-k)m}. \label{phimn}
\end{align}
Now, to choose an element of $T_k$, we can first choose the locations of the $k$ fixed points of $\pi$ in $\{1,\ldots,m\}$, and then choose the remaining part of $\pi$. The first task can be done in ${m\choose k}$ ways, and having done the first task, the second task can be done in $\le (n-k)!$ ways. Thus,
\[
|T_k|\le {m\choose k}(n-k)!. 
\]
We will now use the above to get an upper bound for $k^{\textup{th}}$ term in \eqref{phimn}. First, suppose that $2m-n\le k\le m$ (noting that $2m-n\ge 0$, since $m\ge 2n/3$). Then $n-k\le 2(n-m)$, and hence
\begin{align*}
|T_k| e^{-C_2(m-k)m} &\le {m \choose k} (n-k)!e^{-C_2(m-k)m}\\
&\le {m \choose k} (2(n-m))^{2(n-m)}e^{-C_2(m-k)m},
\end{align*}
interpreting $0^0=1$ if $n=m$. 
This gives
\begin{align}
&\sum_{2m-n \le k\le m} |T_k| e^{-C_2(m-k)m} \notag \\
&\le\sum_{2n-m\le k\le m} {m \choose k} (2(n-m))^{2(n-m)}e^{-C_2(m-k)m}\notag \\
&\le (2(n-m))^{2(n-m)} \sum_{k=0}^m{m \choose m-k} e^{-C_2(m-k)m}\notag \\
&= (2(n-m))^{2(n-m)}(1+e^{-C_2m})^m\le C_3 e^{C_4(n-m)\log (n-m)}.\label{t1bd}
\end{align}
Next, suppose that $0\le k<2m-n$. Then $n-k < 2(m-k)$, which gives
\begin{align*}
|T_k| e^{-C_2(m-k)m} &\le {m \choose k} (n-k)!e^{-C_2(m-k)m}\\
&\le {m \choose k} (2(m-k))^{2(m-k)}e^{-C_2(m-k)m}\\
&\le {m \choose k} e^{C_5(m-k)\log m - C_2 (m-k) m}. 
\end{align*}
This shows that there is a sufficiently large number $n_0$, such that if $n\ge n_0$ (which implies that $m\ge 2n_0/3$), and $0\le k < 2m-n$, then we have
\[
|T_k| e^{-C_2(m-k)m} \le {m\choose k} e^{-C_6(m-k)m}. 
\]
Therefore, 
\begin{align}
\sum_{0\le k< 2m-n} |T_k| e^{-C_2(m-k)m}  &\le \sum_{0\le k< 2m-n} {m\choose m-k} e^{-C_6(m-k)m}\notag \\
&\le (1+e^{-C_6m})^m\le C_7. \label{t2bd}
\end{align}
Combining \eqref{phimn}, \eqref{t1bd}, and \eqref{t2bd}, we get the desired upper bound for sufficiently large $n$. We then get it for all $n$ by just increasing the value of $K_1$.
\end{proof}

We now start towards our proof of Theorem \ref{thm1}. Take any $N$, and let $a = 4/\log 2$, $b = -2/\log2$ and $c := \frac{1}{2}a(1-\log a)$. An easy verification shows that
\begin{align}\label{knuth}
&a \log N + b\log \log N + c \notag \\
&= 4\log_2N - 2\log_2 \log_2N - 2\log_2(4/e) = x_N - 1.
\end{align}
Choose any integer $n$ so that $|x_N - n|\le 2$, and write $n$ as
\begin{align*}
n = a\log N + b\log \log N + d,
\end{align*}
so that 
\begin{align}\label{cdeq}
c+1 - d = x_N - n.
\end{align}
Let $X(i,j)$ and $Y(i,j)$ be the indicators that $\{i,j\}$ is an edge in $\Gamma_1$ and $\Gamma_2$, respectively. Let $\ma$ be the set of all ordered $n$-tuples of distinct numbers from $\{1,\ldots,N\}$. For $A, B\in \ma$, we will write $A\simeq B$ if $X(a_i,a_j)=Y(b_i,b_j)$ for all $1\le i<j\le n$. Let
\[
W := |\{A, B\in \ma: A\simeq B\}|. 
\]
Note that $L_N \ge n$ if and only if $W>0$. We will prove an upper bound on $\pp(W>0)$ using the first moment method, and a lower bound using the second moment method. In the following, we adopt the convention that for any function $f$, $O(f(N))$ denotes any quantity whose absolute value is bounded above by a constant times $f(N)$, where the constant has no dependence on $N$. 
\begin{lmm}\label{xmomlmm1}
Let all notation be as above. Then
\[
\pp(W>0) \le e^{2(c+1-d)\log N + O((\log \log N)^2)}.
\]
\end{lmm}
\begin{proof}
First, note that 
\begin{align*}
\ee(W) &= |\ma|^2 2^{-{n\choose 2}} \le N^{2n} 2^{-{n\choose 2}}. 
\end{align*}
Next, note that if $A\simeq B$, then $A_\pi \simeq B_\pi$ for any $\pi\in S_n$, where $A_\pi$ and $B_\pi$ denote the lists $(a_{\pi(1)},\ldots a_{\pi(n)})$ and $(b_{\pi(1)},\ldots,b_{\pi(n)})$, respectively. Thus, $W>0$ if and only if  $W\ge n!$. This gives 
\[
\pp(W>0) = \pp(W\ge n!) \le \frac{\ee(W)}{n!}\le N^{2n} 2^{-{n\choose 2}}n^{-(n+\frac{1}{2})} e^{n + O(1)}.
\]
Plugging in the value of $n$, we get
\begin{align*}
\pp(W>0) &\le \exp\biggl(2(a\log N + b\log \log N + d) \log N \\
& - \frac{1}{2}(a\log N + b\log \log N + d)(a\log N + b\log \log N +d -1)\log 2\\
& -\biggl(a \log N + b\log \log N + d + \frac{1}{2}\biggr)\log(a \log N + b\log \log N + d)\\
& + a \log N + b\log \log N + d +  O(1)\biggr).
\end{align*}
The third line in the above display is a little bit more complicated than the rest. To simplify, let us use $\log(1+x)=x+O(x^2)$, which gives 
\begin{align*}
&\biggl(a \log N + b\log \log N + d + \frac{1}{2}\biggr)\log(a \log N + b\log \log N + d) \\
&=\biggl(a \log N + b\log \log N + d + \frac{1}{2}\biggr)\biggl(\log(a \log N)+\log\biggl(1 + \frac{b\log \log N + d}{a\log N}\biggr)\biggr) \\
&= \biggl(a \log N + b\log \log N + d + \frac{1}{2}\biggr)(\log\log N + \log a) + b\log \log N + O(1).
\end{align*}
Plugging this into the previous display, let us compute the coefficients of various terms. First, note that the coefficient of $(\log N)^2$ is 
\[
2a - \frac{a^2\log 2}{2},
\]
which is zero since $a=4/\log 2$. Next, the coefficient of $(\log N) \log \log N$ is 
\[
2b - ab\log 2 -a,
\]
which, again, is zero since $a = 4/\log2$ and $b = -2/\log 2$. The next highest term is $\log N$, whose coefficient is
\begin{align*}
2d -\frac{a}{2}(2d-1)\log 2 - a\log a + a &= 2 - a\log a + a - 2d\\
&= 2+2c-2d,
\end{align*}
since $a = 4/\log 2$ and $c = \frac{1}{2}a(1-\log a)$. 
All other terms are of order $(\log \log N)^2$ or smaller. This completes the proof.
\end{proof}
Next, we get a lower bound for $\pp(W>0)$ using the second moment method. For that, we need an upper bound on $\ee(W^2)$. Let $\ma_0$ be the set of all pairs $(A,B)\in \ma^2$ such that $A\cap B=\emptyset$ (considering $A$ and $B$ as sets rather than $n$-tuples). For each $1\le m\le n$, each $1\le i_1<\cdots<i_m\le n$, and each $m$-tuple of distinct $j_1,\ldots,j_m\in \{1,\ldots,n\}$, let $\ma_{i_1,\ldots,i_m; j_1,\ldots,j_m}$ be the set of all $(A,B)\in \ma^2$ such that $a_{i_1}=b_{j_1}, \ldots,a_{i_m} = b_{j_m}$, and $a_i\ne b_j$ for all $i\notin \{i_1,\ldots,i_m\}$ and $j\notin \{j_1,\ldots,j_m\}$. Note that these sets are disjoint, and their union, together with $\ma_0$, equals $\ma$. For $A,B,C,D\in \ma$, let
\[
P(A,B,C,D) := \pp(A\simeq B, \, C\simeq D).
\]
Then note that 
\begin{align*}
&\ee(W^2) = \sum_{A,B,C,D\in \ma} P(A,B,C,D)\\
&= \sum_{(B,D)\in \ma_0} \sum_{A, C\in \ma} P(A,B,C,D) \notag\\
&+ \sum_{m=1}^n \sum_{1\le i_1<\cdots<i_m\le n}\sum_{\substack{1\le j_1,\ldots,j_m\le n\\ \textup{distinct}}} \sum_{(B,D)\in \ma_{i_1,\ldots,i_m;j_1,\ldots,j_m}} \sum_{A,C\in \ma}P(A,B, C,D). \label{xew2}
\end{align*}
Let $\pp'$ denote the conditional probability given $\Gamma_1$. If $(B,D)\in \ma_0$, then 
\[
\pp'(A\simeq B,\, C\simeq D) = 2^{-2{n\choose 2}},
\]
and thus, the unconditional probability is also the same. Next, for $(B,D)\in \ma_{i_1,\ldots,i_m;j_1,\ldots,j_m}$, independence of edges implies that
\begin{align*}
&\pp'(A\simeq B,\, C\simeq D) \\
&= \pp'(X(a_i, a_j) = Y(b_i,b_j) \text{ and } X(c_i,c_j) = Y(d_i,d_j) \textup{ for all } 1\le i<j\le n)\\
&= 2^{-2\bigl({n\choose 2} - {m\choose 2}\bigr)}\pp'(X(a_{i_p},a_{i_q}) = Y(b_{i_p}, b_{i_q}) \text{ and } X(c_{j_p},c_{j_q}) = Y(d_{j_p}, d_{j_q})\\
&\qquad \qquad \qquad \qquad \qquad \qquad \text{ for all } 1\le p<q\le m)\\
&= 2^{-2\bigl({n\choose 2} - {m\choose 2}\bigr)}\pp'(Y(b_{i_p}, b_{i_q}) = X(a_{i_p},a_{i_q})=X(c_{j_p},c_{j_q}) \\
&\qquad \qquad \qquad \qquad \qquad \qquad \text{ for all } 1\le p<q\le m)\\
&= 2^{-2{n\choose 2} + {m\choose 2}} \mathbb{I}_{\{X(a_{i_p},a_{i_q})=X(c_{j_p},c_{j_q}) \text{ for all } 1\le p<q\le m\}},
\end{align*}
where $\mathbb{I}_E$ denotes the indicator of an event $E$, and in going from the third to the fourth line we used the fact that $b_{i_p}=d_{j_p}$ for each $p$. Thus, 
\begin{align*}
P(A,B, C,D)&= 2^{-2{n\choose 2} + {m\choose 2}} \pp(X(a_{i_p},a_{i_q})=X(c_{j_p},c_{j_q}) \\
&\qquad \qquad \qquad \qquad \qquad \text{ for all } 1\le p<q\le m).
\end{align*}
Let $p(a_{i_1},\ldots,a_{i_m}; c_{j_1},\ldots, c_{j_m})$ denote the probability on the right. Combining the above observations, we get
\begin{align*}
\ee(W^2) &= |\ma_0||\ma|^2 2^{-2{n\choose 2}}\\
&\quad + \sum_{m=1}^n \sum_{1\le i_1<\cdots<i_m\le n}\sum_{\substack{1\le j_1,\ldots,j_m\le n\\ \textup{distinct}}}\sum_{A,C\in \ma}  | \ma_{i_1,\ldots,i_m;j_1,\ldots,j_m}|2^{-2{n\choose 2} + {m\choose 2}}\\
&\qquad \qquad \qquad \qquad \qquad \qquad \qquad \qquad \cdot p(a_{i_1},\ldots,a_{i_m}; c_{j_1},\ldots, c_{j_m}).
\end{align*}
Note that 
\begin{align*}
|\ma_0|\le N^{2n}, \ \ |\ma|\le N^{n}, \ \   |\ma_{i_1,\ldots,i_m; j_1,\ldots,j_m}|\le N^{2n-m}.
\end{align*}
Now, given $m$, $i_1,\ldots, i_m$, $j_1,\ldots,j_m$ and $A$, note that by symmetry, 
\begin{align*}
&\sum_{C\in \ma} p(a_{i_1},\ldots,a_{i_m}; c_{j_1},\ldots, c_{j_m})\\
&=\sum_{C\in \ma} p(1,\ldots, m; c_{j_1},\ldots, c_{j_m}) \\
&= (N-m)_{n-m}\sum_{\substack{1\le e_1,\ldots,e_m \le N\\\textup{distinct}}} p(1,\ldots, m; e_{1},\ldots, e_{m}),
\end{align*}
where we used the standard notation $(x)_y = x(x-1)\cdots(x-y+1)$. 
Let $\xi(m,N)$ denote the last sum. Then, using the above information in the expression for $\ee(W^2)$ displayed above, we get 
\begin{align}\label{w2upper}
\ee(W^2) &\le N^{4n}2^{-2{n\choose 2}} \biggl(1+ \sum_{m=1}^n N^{-2m} 2^{{m\choose 2}}{n\choose m}(n)_m\xi(m,N)\biggr). 
\end{align}
Our goal now is to get an upper bound for $\xi(m,N)$. Take any $1\le m\le n$, and any distinct $1\le e_1,\ldots, e_m\le N$. Let $l := |\{e_1,\ldots,e_m\}\cap \{1,\ldots,m\}|$ and let $1\le p_1<\cdots < p_l\le m$ be the indices such that $e_{p_i}\in \{1,\ldots,m\}$ for each $i$. Let $\tilde{\pp}$ denote the conditional probability given $(X(i,j))_{1\le i<j\le m}$. Then
\begin{align*}
&\tilde{\pp}(X(p,q) = X(e_{p}, e_{q}) \text{ for all } 1\le p<q\le m)\\
&= 2^{-{m\choose 2} +{l\choose 2}} \mathbb{I}_{\{X(p_r,p_{s}) = X(e_{p_r}, e_{p_{s}}) \text{ for all } 1\le r<s\le l\}}.
\end{align*}
In the sum defining $\xi(m,N)$,  $e_1,\ldots, e_m$ can be chosen as follows. First, we choose $l$ between $0$ and $m$. Then, given $l$, we choose $1\le p_1<\cdots< p_l\le m$. Given $p_1,\ldots,p_l$, we choose distinct numbers $e_{p_1},\ldots, e_{p_l}\in \{1,\ldots,m\}$. Finally, we choose the rest of the $e_i$'s from outside $\{1,\ldots,m\}$ so that they are distinct. Breaking up the sum in this manner (and rewriting $f_1=e_{p_1},\ldots, f_l = e_{p_l}$, and using symmetry to replace $p_i$ by $i$ for $i=1,\ldots,l$), we get
\begin{align*}
\xi(m,N) &\le \sum_{l=0}^m 2^{-{m\choose 2} +{l\choose 2}}N^{m-l} {m\choose l}\biggl(\sum_{\substack{1\le f_1,\ldots, f_l\le m\\\textup{distinct}}}p(1,\ldots, l; f_{1},\ldots, f_{l})\biggr).
\end{align*}
By Proposition \ref{phiprop}, the inner sum is bounded by $K_1 e^{K_2(m-l)\log (m-l)}$ if $l\ge 2m/3$. If $l< 2m/3$, it is trivially bounded by $m^l$. Combining, we get that the sum is bounded by $C_1e^{C_2 \min\{l,m-l\}\log m}$ for some universal constants $C_1$ and $C_2$, which is bounded by $C_1e^{C_2 \min\{l,n-l\}\log n}$ since $n\ge m$. The same bound holds for ${m\choose l}$. Plugging these into the above display, and then using the resulting bound on $\xi(m,N)$ in \eqref{w2upper}, we get 
\begin{align*}
\ee(W^2) 
&\le N^{4n}2^{-2{n\choose 2}} \biggl(1+ \sum_{l=0}^n \sum_{m=\max\{1,l\}}^n N^{-m-l} 2^{{l\choose 2}}{n\choose m}(n)_m  \\
&\qquad \qquad \qquad \qquad \qquad \cdot C_1 e^{C_2\min\{l,n-l\}\log n}\biggr).
\end{align*}
In the following, we will use $C_1,C_2,\ldots$ to denote arbitrary universal constants, whose values may change from line to line. Throughout, we will implicitly assume that $N$ is large enough (depending only on our choice of $d'$), wherever required. For the innermost sum above, consider three cases. First, take $l=0$. Then the sum over $m$ is bounded by
\begin{align*}
\sum_{m=1}^n N^{-m} n^{2m} C_1 \le C_2n^2 N^{-1}. 
\end{align*}
Next, take $1\le l\le 2n/3$. Then the sum over $m$ is bounded by
\begin{align*}
\sum_{m=l}^n N^{-m-l} n^{2m} 2^{{l\choose 2}}C_1e^{C_2l\log n} \le C_3 n^{2l} N^{-2l}2^{{l\choose 2}} e^{C_2l\log n}\le C_4e^{-C_5l\log N},
\end{align*}
where the last inequality was obtained using 
\begin{align*}
N^{-2l} 2^{{l\choose 2}} &\le e^{-2l\log N} 2^{ln/3} \le e^{-2l\log N} 2^{(5l/3)\log_2 N},
\end{align*}
which holds because $l\le 2n/3$, and $n\le 5\log_2 N$ when $N$ is large enough. Next, using that $n\ge \frac{7}{2}\log_2 N$ for $N$ large enough, we have that for any $l \ge 2n/3$,
\begin{align*}
\frac{N^{-2(l+1)}2^{{l+1\choose 2}}}{N^{-2l}2^{{l\choose2}}} &= N^{-2} 2^l\ge N^{-2} 2^{(7/3)\log_2 N} = N^{1/3},
\end{align*} 
which implies, by backward induction on $l$, that 
\begin{align*}
N^{-2l}2^{{l\choose2}} &\le N^{-(n-l)/3} N^{-2n} 2^{{n\choose 2}}. 
\end{align*}
Thus, for $l\ge 2n/3$, the sum over $m$ is bounded by
\begin{align*}
\sum_{m=l}^n N^{-m-l} 2^{{l\choose 2}} {n\choose l} n! C_1 e^{C_2 (n-l)\log n}&\le C_3 n!N^{-2l} 2^{{l\choose 2}} e^{C_4(n-l)\log n}\\
&\le C_5 n! N^{-2n} 2^{{n\choose 2}}e^{-C_6(n-l)\log N},
\end{align*}
where, in the second inequality, we used the previous display and the fact that $\log N \gg \log n$ when $N$ is large enough. Now, from the proof of Lemma \ref{xmomlmm1}, we have that
\[
n! N^{-2n} 2^{{n\choose 2}} = e^{-2(c+1-d)\log N + O((\log \log N)^2)}.
\]
Combining all of the above (and trivially bounding $e^{-C_6(n-l)\log N}\le 1$), we get
\begin{align*}
\ee(W^2) &\le N^{4n}2^{-2{n\choose 2}}(1 +  e^{-2(c+1-d)\log N + O((\log \log N)^2)}). 
\end{align*}
On the other hand, 
\begin{align*}
\ee(W) &= ((N)_n)^2 2^{-{n\choose 2}}\ge N^{2n}2^{-{n\choose 2}}\biggl(1-\frac{n}{N}\biggr)^{2n}\\
&\ge N^{2n}2^{-{n\choose 2}}\biggl(1-\frac{2n^2}{N}\biggr).
\end{align*}
Thus, by the second moment inequality,
\begin{align*}
\pp(W>0) &\ge \frac{(\ee(W))^2}{\ee(W^2)}\\
&\ge \frac{(1-2n^2/N)^2}{1 +  e^{-2(c+1-d)\log N + O((\log \log N)^2)}}.
\end{align*}
Let $\ve_N := (4\log_2 N)^{-1/2}$, as in the statement of Theorem \ref{thm1}. By Lemma \ref{xmomlmm1}, $\pp(W>0)\to 0$ if $c + 1 -d < -\ve_N$, and by the above lower bound, $\pp(W>0)\to 1$ if $c+1 -d >\ve_N$. But by \eqref{cdeq}, $c+1 -d = x_N - n$. Since $\ve_N\in (0,1/2]$, this proves Theorem \ref{thm1}. 

\begin{remark}
To prove or disprove that $L_N$ concentrates on one point instead of two, one needs to carefully analyze and refine the $O((\log \log N)^2)$ error term in the above analysis and replace it by some explicit term plus $o(1)$ error, so that when $x_N$ is within $O((\log N)^{-1})$ of some integer $n$, one can show that $\pp(W>0)$ is close to neither $1$ or $0$. 
\end{remark}


\section{Subgraph isomorphism}\label{thm2proof}
Deciding if a graph $\Gamma_1$ appears as an induced subgraph of $\Gamma_2$ is a basic problem in the world of image analysis (does this person appear in this crowd scene?), chemistry, and database query. The problem is NP complete but modern constraint satisfaction algorithms can handle $\Gamma_1$ with hundreds of nodes and $\Gamma_2$ with thousands. A comprehensive review of programs and applications is in \cite{mc}. The forthcoming book by Knuth~\cite{knuth} features subgraph isomorphism as a basic problem of constraint satisfaction. 

Comparing algorithms requires a suite of test problems. The authors of \cite{mc} noted that most tests were done on `easy cases' where $\Gamma_2$ is fixed and  $\Gamma_1$ is chosen by choosing a random set of $n$ vertices of $\Gamma_2$ and taking that induced subgraph (so, $\Gamma_1$ appears in $\Gamma_2$). They noticed that taking $\Gamma_1$, $\Gamma_2$ from independent copies of $G(n,p)$, $G(N,q)$ led to different recommendations and conclusions.

As part of their extensive tests they fixed $N=150$ and discovered the phase transition discussed in the introduction. Their results are much richer when $p$ and $q$ are varied --- we only treat $p=q=1/2$. We believe the techniques introduced in this paper will allow similar limit theorems (at least for $p, q$ away from $\{0,1\}$). 

As mentioned earlier in the introduction, a closely related recent paper is that of Alon~\cite{alon}, which shows that if ${N\choose n} 2^{-{n\choose 2}} = \lambda$, then with probability $(1-e^{-\lambda})^2 + o(1)$, $G(N,1/2)$ contains every graph on $n$ vertices as an induced subgraph. Theorem \ref{thm2} is also related to the classical result about the concentration of the size of the largest clique in $G(N,1/2)$, due to Matula~\cite{matula1, matula2} and Bollob\'as and Erd\H{o}s~\cite{be76}. 


Take any $1\le n\le N$. Let $\Gamma_1$ and $\Gamma_2$ be independent $G(n,1/2)$ and $G(N,1/2)$ random graphs. In the remainder of this section, we prove Theorem \ref{thm2}. The proof is similar to that of Theorem \ref{thm1} (in particular, we use Proposition \ref{phiprop}), although a bit simpler because the expected number of copies of $\Gamma_1$ in $\Gamma_2$ has a simpler expression that the expected number of isomorphic pairs of induced subgraphs in two independent random graphs (which has an extra $n!$ in the denominator, leading to the $\log \log N$ correction). 

Let $X(i,j)$ be the indicator that $\{i,j\}$ is an edge in  $\Gamma_1$, and $Y(i,j)$ be the indicator that $\{i,j\}$ is an edge in $\Gamma_2$. Let $\ma$ be the set of all ordered $n$-tuples of distinct numbers from $\{1,\ldots,N\}$, as in the previous section. We will write $A\simeq \Gamma_1$ if $Y(a_i,a_j) = X(i,j)$ for all $1\le i<j\le n$. Let 
\[
W := |\{A\in \ma: A\simeq \Gamma_1\}|. 
\]
Note that $\Gamma_2$ contains a copy of $\Gamma_1$ as an induced subgraph if and only if $W>0$. 
\begin{lmm}\label{momlmm1}
Suppose that $n = a\log N + b$ for $a= 2/\log 2$ and some $b\in \rr$. Then 
\[
\pp(W >0) \le N^{1-b} 2^{-b(b-1)/2}.
\]
\end{lmm}
\begin{proof}
Note that 
\begin{align*}
\ee(W) &= |\ma| 2^{-{n\choose 2}} \le N^n 2^{-{n\choose 2}}. 
\end{align*}
Plugging in $n = a \log N + b$, this gives
\begin{align*}
\ee(W) &\le \exp\biggl((a\log N + b) \log N - \frac{1}{2}(a\log N + b)(a\log N + b -1)\log 2\biggr)\\
&= N^{1-b} 2^{-b(b-1)/2}.
\end{align*}
By Markov's inequality, $\pp(W> 0 ) = \pp(W\ge 1) \le \ee(W)$. This completes the proof.
\end{proof}
For each $(A,B)\in \ma^2$, let 
\[
P(A,B) := \pp(A\simeq \Gamma_1, \, B\simeq \Gamma_1).
\]
Let $\ma_0$ and  $\ma_{i_1,\ldots,i_m; j_1,\ldots,j_m}$ be as in the previous section. Then 
\begin{align}
&\ee(W^2) = \sum_{A,B\in \ma}P(A,B)\notag\\
&= \sum_{(A,B)\in \ma_0}  P(A,B) \notag\\
&+ \sum_{m=1}^n \sum_{1\le i_1<\cdots<i_m\le n}\sum_{\substack{1\le j_1,\ldots,j_m\le n\\ \textup{distinct}}} \sum_{(A,B)\in \ma_{i_1,\ldots,i_m;j_1,\ldots,j_m}} P(A,B). \label{ew2}
\end{align}
We will now use the above expression to establish an upper bound for $\ee(W^2)$. 
\begin{lmm}\label{plmm}
We have
\begin{align*}
\ee(W^2) &\le N^{2n} 2^{-2{n\choose 2}}\biggl(1+ \sum_{m=1}^n \sum_{\substack{1\le j_1,\ldots,j_m\le n\\ \textup{distinct}}} {n\choose m} N^{-m}2^{ {m\choose 2}}p(j_1,\ldots,j_m)\biggr),
\end{align*}
where 
\begin{align*}
p(j_1,\ldots,j_m) := \pp(X(p,q)=X(j_p,j_q) \textup{ for all } 1\le p<q\le m).
\end{align*}
\end{lmm}
\begin{proof}
Let $\pp'$ denote the conditional probability given $\Gamma_1$. If $(A,B)\in \ma_0$, then 
\[
\pp'(A\simeq \Gamma_1,\, B\simeq \Gamma_1) = 2^{-2{n\choose 2}},
\]
and thus, the unconditional probability is also the same. Next, for $(A,B)\in \ma_{i_1,\ldots,i_m;j_1,\ldots,j_m}$, we have
\begin{align*}
&\pp'(A\simeq \Gamma_1,\, B\simeq \Gamma_1) \\
&= \pp'(Y(a_i, a_j) = Y(b_i,b_j) = X(i,j) \textup{ for all } 1\le i<j\le n)\\
&= 2^{-2\bigl({n\choose 2} - {m\choose 2}\bigr)}\pp'(Y(a_{i_p}, a_{i_q}) = X(i_p,i_q)=X(j_p,j_q) \text{ for all } 1\le p<q\le m)\\
&= 2^{-2{n\choose 2} + {m\choose 2}} \mathbb{I}_{\{X(i_p,i_q)=X(j_p,j_q) \text{ for all } 1\le p<q\le m\}},
\end{align*}
where $\mathbb{I}_E$ denotes the indicator of an event $E$. 
Thus, 
\begin{align*}
P(A,B)&= 2^{-2{n\choose 2} + {m\choose 2}} \pp(X(i_p,i_q)=X(j_p,j_q) \text{ for all } 1\le p<q\le m).
\end{align*}
Combining the above observations, we get
\begin{align*}
&\ee(W^2) \\
&= |\ma_0| 2^{-2{n\choose 2}}+ \sum_{m=1}^n \sum_{ i_1<\cdots<i_m}\sum_{\substack{j_1,\ldots,j_m\\ \textup{distinct}}}| \ma_{i_1,\ldots,i_m;j_1,\ldots,j_m}|2^{-2{n\choose 2} + {m\choose 2}}\\
&\qquad \qquad \qquad \qquad \qquad \cdot \pp(X(i_p,i_q)=X(j_p,j_q) \text{ for all } 1\le p<q\le m).
\end{align*}
Note that 
\begin{align*}
|\ma_0|\le N^{2n}, \ \ \   |\ma_{i_1,\ldots,i_m; j_1,\ldots,j_m}|\le N^{2n-m}.
\end{align*}
Plugging these bounds into the previous display and using symmetry, we get the desired result.
\end{proof}
Henceforth, let us assume that 
\begin{align}\label{nN}
n = a\log N + b \ \text{ for }\  a=\frac{2}{\log 2} \ \text{ and some }\  b\in [-1,1].
\end{align}
\begin{lmm}\label{mainbdlmm}
Assume \eqref{nN}, and suppose that $m =\alpha n$ for some $\alpha \in [0,1]$. Then 
\[
-m \log N + {m\choose 2} \log2 \le -\alpha(1-\alpha)n\log N - \alpha(1-\alpha b)\log N + \log 2.
\]
\end{lmm}
\begin{proof}
Note that
\begin{align}
&-m \log N + {m\choose 2}\log 2 \notag \\
&= -\alpha a (\log N)^2 - \alpha b\log N + \frac{(\alpha a \log N + \alpha b) (\alpha a \log N + \alpha b-1)}{2}\log 2\notag\\
&= -\alpha(1-\alpha) a(\log N )^2 + \alpha((2\alpha-1)b - 1) \log N +\frac{\alpha b(\alpha b - 1)}{2}\log 2\notag\\
&\le -\alpha (1-\alpha)(a\log N + b) \log N - \alpha(1-\alpha b) \log N + \log 2,\notag
\end{align}
where in the last inequality we used  $\alpha b(\alpha b -1)\le 2$, which holds because $b\in [-1,1]$ and $\alpha \in [0,1]$.
\end{proof}
\begin{lmm}\label{w2lmm}
There are positive universal constants $C_1$, $C_2$ and $N_0$, such that if $N\ge N_0$, and \eqref{nN} holds, then 
\[
\ee(W^2) \le N^{2n}2^{-2{n\choose 2}} (1+C_1 N^{-C_2 (1-b)}).
\]
\end{lmm}
\begin{proof}
In this proof, $C_1,C_2,\ldots$ will denote arbitrary positive universal constants. Let $p(j_1,\ldots,j_m)$ be as in Lemma \ref{plmm}. First, note that 
\begin{align*}
\sum_{\substack{j_1,\ldots,j_m\\ \textup{distinct}}}p(j_1,\ldots,j_m) &=\frac{\phi(m,n)}{ (n-m)!},
\end{align*}
where $\phi(m,n)$ is the quantity defined in \eqref{phidef}. Therefore, by Proposition \ref{phiprop},
\begin{align*}
&\sum_{2n/3\le m\le n} \sum_{\substack{j_1,\ldots,j_m\\ \textup{distinct}}} {n\choose m} N^{-m}2^{ {m\choose 2}}p(j_1,\ldots,j_m)\\
&\le C_1 \sum_{2n/3\le m\le n} N^{-m}2^{ {m\choose 2}}e^{C_2(n-m)\log n}.
\end{align*}
For $m \ge 2n/3$, Lemma \ref{mainbdlmm} gives (using $(1-\alpha)n = n-m$ and $\alpha b \le (5b+1)/6$ for all $\alpha\in [2/3,1]$ and $b\in [-1,1]$)
\begin{align*}
-m \log N + {m\choose 2}\log 2 &\le -\alpha(n-m)\log N - \frac{5\alpha}{6}(1-b)\log N + \log 2\\
&\le -\frac{2}{3}(n-m)\log N - \frac{5}{9}(1-b)\log N +\log 2. 
\end{align*}
Thus, for $N\ge N_0$, where $N_0$ is a sufficiently large universal constant, we have
\begin{align*}
&\sum_{2n/3\le m\le n} N^{-m}2^{{m\choose 2}}e^{C_2(n-m)\log n} \\
&\le \sum_{2n/3\le m\le n} 2e^{-C_3(n-m)\log N}N^{-5(1-b)/9}\le C_4 N^{-5(1-b)/9}. 
\end{align*}
If $1\le m< 2n/3$, then Lemma \ref{mainbdlmm} gives  (using $\alpha n = m$ and $\alpha(1-\alpha b)\ge 0$) 
\begin{align*}
-m \log N + {m\choose 2}\log 2  &\le -(1-\alpha) m\log N + \log 2\\
&\le -\frac{1}{3} m\log N +\log 2. 
\end{align*}
Thus, using the trivial bound $p(j_1,\ldots,j_m)\le 1$ and the assumption \eqref{nN}, we get
\begin{align*}
&\sum_{1\le m< 2n/3} \sum_{\substack{j_1,\ldots,j_m\\ \textup{distinct}}} {n\choose m} N^{-m}2^{ {m\choose 2}}p(j_1,\ldots,j_m)\\
&\le\sum_{1\le m< 2n/3} 2N^{-m/3} n^{2m}\le 2N^{-C_5},
\end{align*}
provided that $N\ge N_0$ for some sufficiently large universal constant $N_0$. Combining all of the above, and applying Lemma \ref{plmm}, we get the required upper bound.
\end{proof}
\begin{lmm}\label{momlmm2}
There are positive universal constants $N_0$, $C_1$ and $C_2$ such that the following is true. If $N\ge N_0$ and \eqref{nN} holds, then
\[
\pp(W\ge 1) \ge 1 - C_1 N^{-C_2(1-b)}.
\]
\end{lmm}
\begin{proof}
Note that
\begin{align*}
\ee(W) &= (N)_n 2^{-{n\choose 2}}\ge N^n2^{-{n\choose 2}}\biggl(1-\frac{n}{N}\biggr)^n\\
&\ge N^n2^{-{n\choose 2}}\biggl(1-\frac{n^2}{N}\biggr).
\end{align*}
Combining this with the upper bound on $\ee(W^2)$ from Lemma \ref{w2lmm}, and the second moment inequality, we get the desired result.
\end{proof}
\begin{proof}[Proof of Theorem \ref{thm2}]
Note that $\lfloor y_N + \ve_N\rfloor +1 \ge 2\log_2 N +1 + \ve_N$. Therefore, by Lemma~\ref{momlmm1}, $P(\lfloor y_N + \ve_N\rfloor +1, N)\to 0$ as $N\to\infty$. Similarly, note that $\lfloor y_N - \ve_N\rfloor\le 2\log_2 N + 1 - \ve_N$. By Lemma \ref{momlmm2}, this shows that  $P(\lfloor y_N - \ve_N\rfloor, N)\to 1$ as $N\to\infty$. 
\end{proof}

\section{Remarks and problems}
Theorem \ref{thm1} does not capture the way the induced isomorphic subgraph varies from $N$ to $N+1$. If $\Gamma_1$ and $\Gamma_2$ are grown by adding fresh vertices one at a time, it may well be that that largest isomorphic subgraph varies quite a bit (eventually becoming disjoint from earlier champions?). This makes the connection with the limiting $R$ more tenuous and seems worth further study. 

It seems natural to ask similar questions for other graph limit models~\cite{lovasz}. In particular, all $G(\infty,p)$ graphs are isomorphic to the Rado graph, for any $p\in (0,1)$. This makes understanding the largest isomorphic induced subgraph of two independent $G(N, p)$ graphs more interesting.

There are a variety of notions of $\Gamma_1$ being contained in $\Gamma_2$. Just isomorphic as a subgraph (without the `induced' constraint)? As labeled graphs? The classic paper~\cite{barrow} relates such problems to the problem of finding maximal cliques.

We have focused on the yes/no question. There are further counting questions --- how many copies of a pick from $G(n,p)$ appear in $G(N,q)$ and how is this number distributed? (See forthcoming work of Surya, Warnke and Zhu for a solution of this problem.) 

Finally, Knuth's treatment~\cite{knuth} treats subgraph isomorphism as a special case of constraint satisfaction problems, and similar questions can be asked.

\end{document}